\documentclass{article}

\usepackage{color}
\usepackage{cite,amsfonts,amssymb,amsmath,mathrsfs,textcomp,amsbsy,amsthm}

\usepackage{mathrsfs} 
\usepackage{enumitem,lineno}
\usepackage{graphicx}
\usepackage[T1]{fontenc}
\usepackage{amsfonts,amsmath,amssymb}
\usepackage{tikz,comment}
\usetikzlibrary{positioning, calc, trees, decorations.markings, patterns, arrows, arrows.meta}
\usepackage{diagmac2}
\usepackage[colorlinks=true,linkcolor=red,citecolor=cyan,urlcolor=blue,breaklinks=true]{hyperref}

\begin{document}

\newtheorem{Theorem}{Theorem}
\newtheorem{Lemma}{Lemma}
\newtheorem{Proposition}{Proposition}
\newtheorem{Corollary}{Corollary}
\newtheorem{Definition}{Definition}
\newtheorem{Example}{Example}
\newtheorem{Remark}{Remark}
\newtheorem{Conjecture}{Conjecture}

\begin{center}
{\large\bfseries Center of distances of ultrametric spaces generated by labeled trees}

\vspace{10pt}

Oleksiy Dovgoshey$^{1,2}$, Olga Rovenska$^{3,*}$

\vspace{6pt}

{\small
$^{1}$ Institute of Applied Mathematics and Mechanics of NASU, Slovyansk, Ukraine;  
\texttt{oleksiy.dovgoshey@gmail.com}\\
$^{2}$ University of Turku, Turku, Finland;  
\texttt{oleksiy.dovgoshey@utu.fi}\\
$^{3}$ Donbas State Engineering Academy, Kramatorsk, Ukraine;  
\texttt{rovenskaya.olga.math@gmail.com}
}

\end{center}

\vspace{10pt}

\noindent\textbf{Abstract.}  
 The center of distances of a metric space $(X,d)$ is the set
\(
C(X)\) of all $t\in \mathbb R^+$ for which the equation $d(x,p)=t$ has a solution for each $p\in X$.
We prove that the equalities $C(X)=\{0\}$ or
\(
C(X)=\{0,\operatorname{diam}X\}
\)
hold if $(X,d)$ is an ultrametric space generated by labeled trees.
The necessary and sufficient conditions under which $\operatorname{diam} X\in C(X)$ are found.

\vspace{8pt}


\vspace{8pt}


\section{Introduction}

Ultrametric spaces generated by non-negative vertex labelings on arbitrary trees were introduced in \cite{Dov2020TaAoG} and subsequently studied in \cite{DK2024DLPSSAG,DK2022LTGCCaDUS}.  
A characterization of totally bounded ultrametric spaces generated by labeled almost rays was provided in \cite{DV2025JMS}. Furthermore, \cite{DR2025USGbLSG} contains a metric characterization of ultrametric spaces generated by labeled star graphs.

It was shown in \cite{DCR2025OJAC} that the compact ultrametric spaces generated by labeled star graphs are the completions of totally bounded ultrametric space generated by decreasingly labeled rays.

In the present paper we investigate the class ${\bf UT}$ of all ultrametric spaces which are isometric to spaces generated by labeled trees. The main goal of the paper is the characterization of the center of distances of ${\bf UT}$-spaces $(X,d)$. 
The notion of center of distances was introduced in~paper \cite{BPW2018} and used for a generalization of John von~Neumann Theorem on permutations of two sequences with the same set of cluster points~\cite{Neumann1935}.  At the moment the authors are not aware of any results describing the structure of the center of distances in the case of ultrametric spaces.

The paper is organized as follows.
The next section contains some definitions and facts from the theory of ultrametric spaces and graph theory. The main results of the paper are proved in Sections~\ref{sec3}--\ref{sec5}.

Theorem~\ref{dggkq} of Section~\ref{sec3} shows that the center of distances $C(X)$ of an arbitrary ${\bf UT}$-space $(X,d)$ satisfies either the equality
\begin{equation}
    \label{x2}
C(X)=\{0,\operatorname{diam} X\}
\end{equation}
or the equality
\begin{equation}
    \label{x1}
C(X)=\{0\}.
\end{equation}

In Section~\ref{sec4} we introduce the concept of centered spheres and use it in Theorem~\ref{shak1} to describe the geometry of ${\bf UT}$-spaces $(X,d)$ which satisfy \eqref{x2}. Theorem~\ref{efhik} of this section  gives a necessary and sufficient condition    under which each open ball of a given ${\bf UT}$-space is a centered  sphere in this space.
Theorem~\ref{kmhg} of Section~\ref{sec4} shows that an ultrametric space $(X,d)$ is equidistant if and only if each centered sphere of $(X,d)$ is an open ball of $(X,d)$.

Theorem~\ref{kgcd} of Section~\ref{sec5} shows that equality~\eqref{x2} holds for ${\bf UT}$-space $(X,d)$ if and only if the diametrical graph of $(X,d)$ has a spanning star subgraph.
Corollary~\ref{shak} of Theorem~\ref{rgjkjs24} shows that equality~\eqref{x1} holds for ${\bf UT}$-space $(X,d)$ with $|X|\geq 2$ if and only if $(X,d)$ is weakly similar to an unbounded ultrametric space.

The final Section~\ref{sec6} contains some conjectures related to results of the paper.

\section{Preliminaries}

In what follows we will use the symbol $\mathbb R^+$ to denote the set $[0,+\infty)$.

\begin{Definition}\label{deff1}Let $X$ be a non-empty set. A metric on $X$ is a symmetric function $d\colon X\times X \to \mathbb R^+$ satisfying the equality 
$d(x,x)=0$ for each $x\in X$, the inequality $d(x,y)>0$ for all distinct $x,y\in X$, and  the triangle inequality
 \[d(x,y)\leq d(x,z) + d(z,y)\]
 for all \(x\), \(y \), \(z \in X\).
\end{Definition}

A metric space \((X, d)\) is called an \emph{ultrametric space} if the \emph{strong triangle inequality}
\[
d(x,y)\leq \max \{d(x,z),d(z,y)\}
\]
holds for all \(x\), \(y\), \(z \in X\). The function \(d\) is called \emph{an ultrametric} on \(X\).

Following \cite[p.~46]{Dez-Dez} we say that a metric $d : X \times X \to \mathbb{R}^+$ is \emph{equidistant} if there is a constant $k > 0$ such that
\[
d(x,y) = k
\]
whenever $x$ and $y$ are distinct points of $X$.

\begin{Definition}\label{eddrr}
Let $(X,d)$ and $(Y,\rho)$ be metric spaces. A bijective mapping
\(
\Phi \colon X \to Y
\)
is said to be an isometry if
\[
d(x,y) = \rho(\Phi(x), \Phi(y))
\]
holds for all $x,y \in X$. The metric spaces are \emph{isometric} if there is an
isometry of these spaces.
\end{Definition}

Let $(X,d)$ be a metric space. In what follows we will denote by $D(X)$ the \textit{distance set} of the space $(X,d)$,
\begin{equation}
    \label{tghii11}
D(X) := \{ d(x,y) : x,y \in X \}.
\end{equation}


\begin{Definition}\label{zcgh572}
Let $(X,d)$ be a metric space.
The center of distances of $(X,d)$ is the set $C(X)$
 defined as
\begin{equation}
    \label{edft34Gb}
C(X):=\{t \in D(X) : \forall p\in X\ \exists x\in X\ d(p,x)=t\}.
\end{equation}
\end{Definition}

\begin{Example}[The $p$-adic ultrametric]\label{ex1}
Let $p \geq 2$ be a prime number, $\mathbb Q$ be the set of rational numbers, and $\mathbb Z$ be the set of integers. The $p$-adic valuation of $t \in \mathbb{Q}$ is defined as
\begin{equation}\label{hh1}
|t|_p :=
\begin{cases}
p^{-\gamma}, & \text{if } t \neq 0, \\
0, & {\text otherwise},
\end{cases}
\end{equation}
where $\gamma=\gamma(t)$ is the unique $\gamma \in \mathbb{Z}$ such that
\[
t=p^\gamma \frac{m}{n},
\]
and $m,$ $n$ are coprime to $p$.
The $p$-adic valuation satisfies the inequality
\begin{equation*}
|t+w|_p \leq \max\{|t|_p,|w|_p\}
\end{equation*}
for all $t,w \in \mathbb{Q}$.
This implies that the mapping
\(
d_p : \mathbb{Q} \times \mathbb{Q} \to \mathbb{R}^+, \)
\begin{equation}
    \label{f}
d_p(t,w) := |t-w|_p
\end{equation}
is an ultrametric on $\mathbb{Q}$. (See formula (1.10) in Chapter~1 of \cite{Bachman1964}.)
It directly follows from \eqref{f} that the mapping
\[
{\mathbb Q} \ni x \mapsto x+y \in {\mathbb Q}
\]
is a self-isometry of the space $({\mathbb Q}, d_p)$ for every $y\in \mathbb Q$. Which gives us the equalities
\begin{equation}
    \label{ff1}
D(\mathbb{Q}) = C(\mathbb{Q}) = \{\, p^\gamma : \gamma \in \mathbb{Z} \,\} \cup \{0\}
\end{equation}
in accordance with formulas \eqref{tghii11}--\eqref{f}.
\end{Example}

Let $A$ be a non-empty subset of a metric space $(X,d)$. The quantity
\begin{equation}\label{lokias}
    \operatorname{diam} A := \sup\{ d(x,y) : x,y \in A \}
\end{equation}
is the \emph{diameter} of $A$. If the inequality $\operatorname{diam} A < \infty$
holds, then we say that $A$ is a \emph{bounded} subset of $(X,d)$.

The next concept was introduced in \cite{DP2013AMH} as a generalization of the concept of similarity for general semimetric spaces.

\begin{Definition}\label{xzsdd}
\label{scguite}
Let $(X, d)$, $(Y, \rho)$ be metric spaces and let $D(X)$, $D(Y)$ be the distance sets of $(X, d)$, $(Y, \rho)$ respectively.
 A bijective mapping $\Phi\colon X \to Y$ is a {\it weak similarity} if there is a strictly increasing bijection $f\colon D(Y) \to D(X)$ such that 
\begin{equation*}
d(x, y) = f \left( \rho \left( \Phi(x), \Phi(y) \right) \right)
\end{equation*}
for all $x, y \in X$.
\end{Definition}

We will say that metric spaces $(X,d)$ and $(Y,\rho)$ are \emph{weakly similar} if there is a weak similarity $
X \to Y $.

The next theorem is a reformulation of Theorem~3.22 from \cite{BDK2022TAG}.

\begin{Theorem}\label{rgjkjs24}
Let $(X,d)$ be an ultrametric space with $|X|\geq 2.$ Then the following statements are equivalent:
 \begin{enumerate}[label=\textit{(\roman*)}, left=0pt]
     \item $ (X,d)$ is weakly similar to an unbounded ultrametric space.
 \item The distance set $D(X)$ does not contain $\operatorname{diam} X$, 
 \begin{equation*}
     \operatorname{diam} X \notin D(X).
 \end{equation*}
 \end{enumerate}
\end{Theorem}

\begin{Lemma}\label{nnhh}
Let $A$ be a non-empty subset of an ultrametric space $(X,d)$. Then the equality
\begin{equation}\label{ssqq}
\operatorname{diam} A=\sup\{d(p,a): a\in A\}
\end{equation}
holds for each $p\in A$.
\end{Lemma}

\begin{proof}
Let $p$ be an arbitrary point of $A$. Then, using the strong triangle inequality, we obtain
\[
d(x,y)\leq \max\{d(p,x),d(p,y)\}\leq \sup\{d(p,a)\colon a\in A\}.
\]
Consequently, the inequality
\[
\operatorname{diam} A \leq \sup\{d(p,a)\colon a\in A\}
\]
holds.
The converse inequality follows directly from~\eqref{lokias}. Thus equality~\eqref{ssqq} holds.
\end{proof}

\begin{Remark}
Equality \eqref{ssqq} is known, see, for example, equality (1.3) in \cite{Dov2019pNUAA}, but the authors cannot give here an exact reference to standard books on ultrametric spaces.
\end{Remark}

Let \((X, d)\) be a metric space. The \emph{open ball} with a \emph{radius} \(r > 0\) and a \emph{center} \(c \in X\) is the set
\begin{equation}
    \label{02}
B_r(c): = \{x \in X \colon d(c, x) < r\}.
\end{equation}
Let us denote by ${\bf B}_X$ the set of all open balls of metric space $(X,d)$.

The next proposition is a part of Proposition 18.4 from \cite{Sch1985}.

\begin{Proposition}
    \label{dj}
 Let $(X,d)$ be an ultrametric space. Then, for every ball $B_r(c) \in {\bf B}_X$ and every $a \in B_r(c)$, we have
\[
B_r(c) = B_r(a).
\]
\end{Proposition}

Now, we recall the concepts of \emph{compactness} and \emph{total boundedness}.

A standard definition of compactness is usually formulated as:
a metric space $(X,d)$ is compact if every open cover of $X$ has a finite subcover.

\begin{Definition}\label{aghj}
A metric space $(X,d)$ is totally bounded if for every
$r>0$ there is a finite set
\(
\{B_r(x_1),\ldots,B_r(x_n)\}\subseteq {\bf B}_X
\)
such that
\[
X \subseteq \bigcup_{i=1}^{n} B_r(x_i).
\]
\end{Definition}

It follows directly from Definition~\ref{aghj} that every compact metric space is totally bounded.

The next proposition is a corollary of Theorem~3.6 from paper \cite{DS2022TRoUCaS}.

\begin{Proposition}\label{zbhfd}
Let $(X,d)$ be an infinite totally bounded ultrametric space and let $D(X)$ be the distance set of the space $(X,d)$. Then there is a 
 strictly decreasing sequence
$(x_n)_{n \in \mathbb{N}}$ of points of $\mathbb R^+$ such that
\[
\lim_{n \to \infty} x_n = 0
\]
 holds and the equivalence
\[
(x \in D(X)) \iff (x = 0 \text{ or } \exists n \in \mathbb{N} : x_n = x)
\]
is valid for every $x \in \mathbb{R}^+$.
\end{Proposition}

Further, let us recall some definitions and facts from graph theory.

A \textit{graph} is a pair $(V, E)$ consisting of a non-empty set $V$ and a set $E$ whose elements are unordered pairs $\{u, v\}$ of different points $u, v \in V$. For a graph $G = (V, E)$, the sets $V = V(G)$ and $E = E(G)$ are called \textit{the set of vertices} and \textit{the set of edges}, respectively. A graph $G$ is \textit{finite} if $V(G)$ is a finite set. If $\{x, y\} \in E(G)$, then the vertices $x$ and $y$ are \emph{ends} of $\{x,y\}$ and called \textit{adjacent}. In what follows, we will always assume that $E(G) \cap V(G) = \emptyset$.
A graph $G$ is \emph{empty} if $E(G)=\emptyset$.

The following concept is well-known for finite graphs  (see, for example, \cite[p.~17]{Diestel2005}). 

\begin{Definition}\label{scfr} Let $G$ be a graph and let $k \geq 2$ be a cardinal number. The graph $G$
is complete $k$-partite if the vertex set $V(G)$ can be partitioned into $k$ non-empty
disjoint parts, in such a way that no edge has both ends in the same part and
any two vertices in different parts are adjacent.
\end{Definition}

We say that $G$ is a \emph{complete multipartite graph} if  $G$ is complete $k$-partite for some $k$.

The next concept was introduced in~\cite{PD2014JMS} for characterization of finite ultrametric spaces for which the Gomory--Hu inequality~\cite{GH1961S} turns to equality.

\begin{Definition}\label{dvhj} Let $(X,d)$ be a metric space. Denote by $G_{X}$ a graph
such that $V(G_{X}) = X$ and, for $u, v \in V(G_{X})$,
\begin{equation}
    \label{ojhr}
(\{u,v\} \in E(G_{X})) \iff (d(u,v) = \operatorname{diam} X \text{ and } u \ne v).
\end{equation}

We call $G_{X}$ the diametrical graph of $(X,d)$.
\end{Definition}

\begin{Remark}\label{suit}
It directly follows  from \eqref{ojhr} that $G_X$ is empty if and only if  $X$ is a singleton or $\operatorname{diam} X \notin D(X)$.
\end{Remark}

The following theorem is a direct consequence of Theorems 3.1 and 3.2 from \cite{DDP2011pNUAA}.

\begin{Theorem}\label{ref} Let $(X,d)$ be an ultrametric space with $|X| \geq 2$. Then the following
statements are equivalent:

\begin{enumerate}[label=\textit{(\roman*)}, left=0pt]
\item The diametrical graph $G_{X}$ is non-empty.
\item The diametrical graph $G_{X}$ is complete multipartite.
\end{enumerate}

\end{Theorem}

Let $G$ be a graph.
A graph \(G_1\) is a \emph{subgraph} of \(G\) if
\[
V(G_1) \subseteq V(G) \quad \text{and} \quad E(G_1) \subseteq E(G).
\]
In this case we will write \(G_1 \subseteq G\).

A \emph{path} is a finite graph \(P\) whose vertices can be numbered without repetitions so that
\begin{equation}\label{e3.3-1}
V(P) = \{x_1, \ldots, x_k\} \quad \text{and} \quad E(P) = \{\{x_1, x_2\}, \ldots, \{x_{k-1}, x_k\}\}
\end{equation}
with \(k \geqslant 2\). We will write \(P = (x_1, \ldots, x_k)\) or \(P = P_{x_1, x_k}\) if \(P\) is a path satisfying \eqref{e3.3-1} and say that \(P\) is a \emph{path joining \(x_1\) and \(x_k\)}. A graph \(G\) is \emph{connected} if for every two distinct vertices of \(G\) there is a path \(P \subseteq G\) joining these vertices.

A finite graph $C$ is a \textit{cycle} if there exists an enumeration of its vertices without repetitions such that $V(C) = \{x_1, \ldots, x_n\}$ and
\begin{equation*}
E(C) = \{\{x_1, x_2\}, \ldots, \{x_{n-1}, x_n\}, \{x_n, x_1\}\} \quad \text{with } n \geq 3.
\end{equation*}

\begin{Definition}
A connected graph $T$  without cycles is called a \textit{tree}. 
\end{Definition}

Let us recall the concept of star graphs.

\begin{Definition}\label{efhjc}
A tree $T$ is  a star graph if there exists a vertex
$c \in V(T)$ such that 
\begin{equation*}
E(T)=\{\{v,c\} \colon v\in V(T)\setminus \{c\}\}.
\end{equation*}
\end{Definition}

A subtree $T$ of a connected graph $G$ is called a \emph{spanning tree} of $G$ if the equality $V(T)=V(G)$ holds.

Let us introduce the concept of labeled trees.

\begin{Definition}\label{d2.4}
A {\it labeled tree} is a pair \((T, l)\), where \(T\) is a tree and \(l\) is a mapping defined on \(V(T)\).
\end{Definition}

If $(T,l)$ is a labeled tree, then we  write \(T = T(l)\) instead of \((T, l)\). Moreover, in what follows, we will consider only the non-negative real-valued labelings \(l\colon V(T)\to \mathbb R^{+}\).

Following~\cite{Dov2020TaAoG}, for arbitrary labeled tree \(T = T(l)\), we define a mapping \(d_l \colon V(T) \times V(T) \to \mathbb R^{+}\) as
\begin{equation}\label{e11.3}
d_l(u, v) := \begin{cases}
\max\limits_{w \in V(P_{u,v})} l(w), & \text{if} \,\,u \neq v,\\
0 ,& \text{otherwise},
\end{cases}
\end{equation}
where \(P_{u,v}\) is a path joining \(u\) and \(v\) in \(T\).

\begin{Theorem}\label{t11.9}
Let \(T = T(l)\) be a labeled tree. The mapping \(d_l\) is an ultrametric  on the set $V(T)$ if and only if the inequality
\begin{equation}\label{t11.9:e1}
\max\{l(u), l(v)\} > 0
\end{equation}
holds for every \(\{u, v\} \in E(T)\).
\end{Theorem}

A proof of Theorem \ref{t11.9} can be obtained by simple modification of the proof of Proposition~3.2 in \cite{Dov2020TaAoG}.

As in \cite{DV2025JMS,DR2025USGbLSG} we shall say that a labeling $l \colon V(T) \to \mathbb{R}^+$ is {\it non-degenerate} if  inequality \eqref{t11.9:e1}
holds for every $\{u, v\} \in E(T)$.

The following lemma will be used in the next section of the paper.

\begin{Lemma}\label{lem1}
Let $T=T(l)$ be a labeled tree with a non-degenerate labeling
$l\colon V(T)\to \mathbb{R}^+$ and let $D(V(T))$ be the distance set of the ultrametric space $(V(T),d_l)$.
Then there exists a non-degenerate labeling
$l^*\colon V(T)\to \mathbb{R}^+$ such that
\begin{equation}
    \label{eq0}
d_{l^*}=d_{l}
\end{equation}
and the equality
\begin{equation}
    \label{eq1}
D(V(T))=\{l^*(u)\,:\,u\in V(T)\}
\end{equation}
holds.
\end{Lemma}

\begin{proof}
Let us define a labeling $l^*\colon V(T)\to \mathbb{R}^+$ as
\begin{equation}
l^*(u):=
\begin{cases}
l(u), & \text{if}\,\, l(u)\in D(V(T)),\\
0, & \text{otherwise}.
\end{cases}
\label{eq:2}
\end{equation}

Equalities \eqref{eq0} and \eqref{eq1} evidently hold if $|V(T)|=1$. Suppose that $|V(T)|\geq 2$ and consider an arbitrary edge $\{u^\circ,v^\circ\}\in E(T)$.
Using \eqref{e11.3}  we obtain
\begin{equation}
0<d_l(u^\circ,v^\circ)=\max\{l(u^\circ),l(v^\circ)\}
\label{eq:3}
\end{equation}
and, consequently,
\begin{equation}
    \label{eq4}
    \max\{l(u^\circ),l(v^\circ)\} \in D(V(T)).
\end{equation}
Membership relation \eqref{eq4} and formula~\eqref{eq:2} imply the equality
\begin{equation}
\max\{l(u^\circ),l(v^\circ)\}=\max\{l^*(u^\circ),l^*(v^\circ)\}.
\label{eq:5}
\end{equation}
The last equality and formula~\eqref{eq:2} show that
$l^*\colon  V(T)\to\mathbb{R}^+$ is a non-degenerate labeling of $V(T)$ because $l \colon V(T) \to \mathbb R^+$ is non-degenerate.
Thus $(V(T),d_{l^*})$ is an ultrametric space by Theorem~\ref{t11.9}.

Let us prove equality~\eqref{eq0}.
Let $x$ and $y$ be two distinct vertices of $T$ and let $P_{x,y}\subseteq T$
be the path joining $x$ and $y$.
Then using~\eqref{e11.3}, \eqref{eq:3}, and \eqref{eq:5} we obtain
\[
d_l(x,y)=\max_{w \in V(P_{x,y})} l(w)
        =\max_{\{u,v\}\in E(P_{x,y})} \{l(u),l(v)\}=\max_{\{u,v\}\in E(P_{x,y})} \{l^*(u),l^*(v)\}=
\]
\[
\max_{w \in V(P_{x,y})} \{ l^{*}(w) \}
=
d_{l^{*}}(x,y).
\]
Thus equality \eqref{eq0} holds.

To prove equality \eqref{eq1} it suffices to note that
equality \eqref{e11.3} with $l=l^*$ implies the inclusion
\[
D(V(T)) \subseteq \{ l^{*}(u) : u \in V(T) \}
\]
and that the converse inclusion
\[
D(V(T)) \supseteq \{ l^{*}(u) : u \in V(T) \}
\]
follows from \eqref{eq:2}.

The proof is completed. 
\end{proof}

In what follows we will denote by ${\bf UT}$ the class of all ultrametric spaces
which are isometric to the space generated by labeled trees 
with non-degenerate vertex labelings.

The main goal of the present paper is to describe the structure of centers of distances
$C(X)$ for $(X,d) \in {\bf UT}$.

\section{Center of distances of UT-spaces}\label{sec3}

Let $(X,d)$ be a metric space,  let $p \in X$, and  
let $D_p(X)$ be a subset of 
the distance set $D(X)$
 defined as
\begin{equation}
    \label{fesdi6}
D_p(X) := \{ d(p,x) : x \in X \}. 
\end{equation}

\begin{Proposition}\label{ugrdd}
Let $(X,d)$ be a metric space and let $C(X)$ be the center of distances of $(X,d)$. Then the equality 
\begin{equation}
    \label{wee344}
C(X) = \bigcap_{p \in X} D_p(X) 
\end{equation}
holds.
\end{Proposition}

\begin{proof}
All metric spaces are non-empty by Definition~\ref{deff1}.
Consequently, the right side of equality \eqref{wee344} is correctly defined.
Now \eqref{wee344} follows from Definition \ref{zcgh572} and formulas \eqref{edft34Gb},  \eqref{fesdi6}. 
\end{proof}

\begin{Proposition}\label{xxxy} Let $(X,d)$ be a metric space and let $C(X)$ be the center of distances of $(X,d)$.
Then  $C(X)$ contains the point $0$.
\end{Proposition}

\begin{proof}
Since we have $d(p,p)=0$ for every $p \in X$, equality \eqref{fesdi6} implies
\[
0 \in D_p(X)
\]
for all $p \in X$.
Hence
\begin{equation*}
0 \in \bigcap_{p \in X} D_p(X)
\end{equation*}
holds. The last membership relation and Proposition~\ref{ugrdd} imply
\(
    0\in C(X)
\)
as required.
\end{proof} 

\begin{Remark}
 Propositions \ref{ugrdd} and \ref{xxxy}  are known, see, for example, \cite{NPT2025} and \cite{BPW2018}.
\end{Remark}

Lemma~\ref{nnhh} and Proposition~\ref{ugrdd} imply the following corollary.

\begin{Corollary}\label{rep}
Let $(X,d)$ be an ultrametric space, $D(X)$ be the distance set of $(X,d)$ and $C(X)$ be the centers of distances of $(X,d)$. If 
$\operatorname{diam} X $ belongs to $D(X)$,
\begin{equation*}
    \operatorname{diam} X \in D(X),
\end{equation*}
then $\operatorname{diam} X$ also belongs to $C(X)$, 
\begin{equation*}
    \operatorname{diam} X \in C(X).
\end{equation*}
\end{Corollary}

\begin{Lemma}\label{iixx}
Let $T = T(l)$ be a labeled tree with a non-degenerate labeling 
\(
l \colon V(T) \to \mathbb{R}^{+}
\) satisfying the equality
\begin{equation}
    \label{er00}
D(V(T)) = \{ l(u) : u \in V(T) \}.
\end{equation}
Then
\begin{equation}
    \label{er01}
    \alpha\notin C(X)
\end{equation}
holds whenever 
\begin{equation}
    \label{er03}
0<\alpha<\operatorname{diam}V(T).
\end{equation}
\end{Lemma}

\begin{proof}

Let us consider an arbitrary $\alpha$ satisfying \eqref{er03}.
Then, using \eqref{lokias}, \eqref{er00}, and \eqref{er03}, we can find a vertex $v^{*}\in V(T)$ such that
\begin{equation}
    \label{ks}
\alpha<l(v^{*}).
\end{equation}
Write
\[
D_{v^{*}}(V(T)):=\{d_{l}(v^{*},u):u\in V(T)\}.
\]
Equality \eqref{e11.3} implies the inequality
\[
d_{l}(v^{*},u)\geq l(v^{*})
\]
for each $u\in V(T)\setminus\{v^{*}\}$.
Hence $\alpha$ does not belong to the set $D_{v^{*}}(V(T))$ by \eqref{ks}.
Consequently \eqref{er01} holds by Proposition~\ref{ugrdd}.
\end{proof}

The next theorem completely describes the center of distances for ultrametric spaces $(X,d)\in{\bf UT}$.

\begin{Theorem}\label{dggkq}
Let $(X,d) \in {\bf UT}$ have at least two points, let $D(X)$ be the distance set of $(X,d)$, and let $C(X)$ be the center of distances of $(X,d)$.
Then
 we have either
\begin{equation}\label{jy1}
    C(X)=\{0\},
\end{equation}
or
\begin{equation}\label{jy2}
    C(X)=\{0,\operatorname{diam}X\}
\end{equation}
and, moreover, the last equality holds if and only if 
\begin{equation}
    \label{er0}
\operatorname{diam} X \in D(X).
\end{equation}


\end{Theorem}

\begin{proof} Let us prove the equivalence \eqref{jy2}
$\Leftrightarrow$ \eqref{er0}.

Since $C(X)$ is a subset of $D(X)$, equality \eqref{jy2} implies membership relation \eqref{er0}.

Suppose now that we have \eqref{er0}. Then Corollary~\ref{rep} implies
\(
\operatorname{diam}X\in C(X).
\)
Proposition~\ref{xxxy} gives us
\(
0\in C(X).
\)
Consequently, the inclusion
\begin{equation}
    \label{tt}
C(X)\supseteq\{0,\operatorname{diam}X\}
\end{equation}
is valid.
Since $(X,d)$ belongs to the class ${\bf UT}$, there exists a labeled tree $T=T(l)$  such that $(X,d)$ and $(V(T),d_l)$ are isometric. By Lemma~\ref{lem1} we can also suppose that
\begin{equation}
    \label{27-5.01}
D(V(T))=\{\,l(u)\colon u\in V(T)\,\}.
\end{equation}
Then using Lemma~\ref{iixx} we obtain
\(
C(V(T)) \subseteq \{0,\operatorname{diam} V(T)\}.
\)
Since $(X,d)$ and $(V(T),d_l)$ are isometric, the last inclusion gives the inclusion
\begin{equation}
    \label{svb}
C(X) \subseteq \{0,\operatorname{diam} X\}.
\end{equation}
Equality~\eqref{jy2} follows from \eqref{tt} and \eqref{svb}. Thus the 
equivalence \eqref{jy2}
$\Leftrightarrow$ \eqref{er0} is valid.

Let us now assume that equality \eqref{jy2} does not hold.
Then, as was shown above, the formula
\begin{equation}
    \label{t1}
\operatorname{diam} X \notin D(V(T)) 
\end{equation}
is valid. Let \(T = T(l)\) be a labeled tree such that \((X,d)\)
and \((V(T),d_l)\) are isometric and equality \eqref{27-5.01} holds.
Then using Proposition~\ref{xxxy}, Lemma \ref{iixx} and \eqref{t1} we see that
\begin{equation}
    \label{t2}
C(V(T)) = \{0\}.
\end{equation}
Since \((X,d)\) and \((V(T),d_l)\)
are isometric, \eqref{jy1} follows from \eqref{t2}.

The proof is completed. 
\end{proof}

\begin{Remark}
If $(X,d)$ contains exactly one point, then we have
\(
\operatorname{diam} X = 0
\)
and, consequently, equalities \eqref{jy1} and \eqref{jy2} are equivalent.
\end{Remark}

\begin{Example}[The Delhommé-Laflamme-Pouzet-Sauer ultrametric]\label{xi}
 Let us define a mapping
\(
d^{+} : \mathbb{R}^{+} \times \mathbb{R}^{+} \to \mathbb{R}^{+}
\)
as
\begin{equation*}
d^{+}(p,q):=
\begin{cases}
\max\{p,q\}, & \text{if } p\neq q,\\
0, & \text{otherwise}.
\end{cases}
\end{equation*}
Then $d^{+}$ is an ultrametric on $\mathbb{R}^{+}$ and the inequality
\[
d^{+}(0,p)\le d(p,q)
\]
holds whenever $0\neq p\neq q$. The last inequality and Theorem~2.1 from \cite{DR2025USGbLSG} imply the existence of a labeled star graph $S=S(l)$ such that $V(S)=\mathbb R^+$ and $d^+=d_l$, where $d_l$ is an ultrametric defined by \eqref{e11.3} with $T=S$. Thus $(\mathbb{R}^+, d^+)$ is an {\bf UT}-space.
Since the ultrametric space $(\mathbb{R}^+, d^+)$ is unbounded, Theorem~\ref{dggkq} implies that the center of distances of $(\mathbb{R}^+, d^+)$ contains the unique point $0$, 
\[
C(\mathbb{R}^+) = \{0\}.
\]
\end{Example}

\begin{Remark}
The ultrametric $d^+$ on $\mathbb{R}^+$ was introduced by Delhommé, Laflamme, Pouzet, and Sauer in \cite{DLPS2008TaiA}.
\end{Remark}

\section{Center of distances and centered spheres}\label{sec4}

The next concept was introduced in \cite{DK2022LTGCCaDUS}.

\begin{Definition}\label{gla83}
Let $(X,d)$ be a metric space.
A set $S \subseteq X$ is a centered sphere in $(X,d)$ if there exist
$c \in S$, the center of $S$, and $r \in \mathbb{R}^+$, the radius of $S$, such that
\begin{equation}
    \label{ijgt}
S = \{x \in X : d(x,c) = r\} \cup \{c\}.
\end{equation}
\end{Definition}

If $r=0$, then \eqref{ijgt} gives the equality $S=\{c\}$. Hence the following proposition holds.

\begin{Proposition}\label{sat}
Let $(X,d)$ be a metric space. Then the singleton $\{c\}$ is a centered sphere in $(X,d)$ for every $c\in X$.
\end{Proposition}

\begin{Theorem}\label{shak1}
Let $(X,d)\in {\bf  UT}$ have at least two points and let $C(X)$ be the center of distances of $(X,d)$.
Then the equality
\begin{equation}
    \label{ks1-1}
C(X)=\{0,\operatorname{diam}X\}
\end{equation}
holds if and only if $X$ is a centered sphere in $(X,d)$ and $\operatorname{diam} X$ is the radius of this sphere.
\end{Theorem}

\begin{proof}
Let $X$ be a centered sphere in $(X,d)$.
Then, by Definition~\ref{gla83}, there are $c\in X$ and $r\in \mathbb{R}^+$ such that
\begin{equation}
    \label{qeszf}
X=\{x\in X : d(x,c)=r\}\cup\{c\}. 
\end{equation}
The last equality and the strong triangle inequality imply
\[
d(x,y)\leq \max\{d(x,c),d(c,y)\}\leq r
\]
for all $x,y\in X$.
Consequently we have
\(
\operatorname{diam} X = r 
\)
and
\(
\operatorname{diam} X \in D(X), 
\)
where $D(X)$ is the distance set of $(X,d)$.
Hence \eqref{ks1-1}  follows from Theorem~\ref{dggkq}.

Now we must show that \eqref{ks1-1}  implies \eqref{qeszf}.
Let us do it.
Since $(X,d)\in\bf UT$ holds, there is a labeled tree $T=T(l)$ such that $(X,d)$ and $(V(T),d_l)$ are isometric.
Consequently, we have the equalities
\[
D(X)=D(V(T)), \quad C(X)=C(V(T))
\]
and can rewrite equality \eqref{ks1-1}  as
\begin{equation}
    \label{rvguq}
C(V(T))=\{0,\operatorname{diam} V(T)\}. 
\end{equation}
By Lemma \ref{lem1} we may also suppose that
\begin{equation}
    \label{er1}
D(V(T))=\{l(u)\colon u\in V(T) \}.
\end{equation}
It follows from \eqref{rvguq} that
the membership relation
\[
\operatorname{diam} V(T) \in D(V(T))
\]
holds, and, consequently, by \eqref{er1} there exists \( u^{*} \in V(T) \) such that
\begin{equation}
    \label{er11}
\operatorname{diam}V(T)=l(u^{*}) \geq l(u) 
\end{equation}
for each \( u \in V(T) \).

Using formulas \eqref{e11.3} and  \eqref{er11}, it is easy to prove that
\[
d_l(u^*,u)=l(u^*)
\]
holds for each $u\in V(T)\setminus\{u^*\}.$
The last statement implies  \eqref{qeszf} with
\(
c=\Phi(u^*) \) and \(r=\operatorname{diam} X,
\)
where $\Phi(u^*)$ is the image of the point $u^*\in V(T)$ under some isometry
\(
\Phi:V(T)\to X
\)
of the spaces $(V(T),d_l)$ and $(X,d)$.

The proof is completed.
\end{proof}

Our next goal is to describe some conditions under which every open ball of a space $(X,d)\in {\bf UT}$ is a centered sphere in this space.

The following lemma is a reformulation of Lemma~4.3 from \cite{DK2024DLPSSAG}.

\begin{Lemma}\label{scgkk}
Let $(X,d)\in \mathbf{UT}$ and let $T=T(l)$ be a labeled tree such that $(X,d)$ and $(V(T),d_l)$ are isometric.
Then, for every $B^{1}\in {\bf B}_{X}$ there is a subtree $T^{1}$ of $T$ such that $
(B^{1},d\big|_{B^{1}\times B^{1}})$
and
$(V(T^1), d_{l^{1}})$ are isometric,
 where $l^{1}$ is the restriction of the labeling
$l : V(T) \to \mathbb{R}^{+}$ on the set $V(T^{1})$.
\end{Lemma}

We will also use the next lemma.
\begin{Lemma}\label{dhiu}
     Let $(X,d)$ be an ultrametric space and let $B$ be an open ball in $(X,d)$. Then $B$ is a centered sphere in $(X,d)$ if and only if this ball is a centered sphere in $(B, d|_{B\times B})$.
\end{Lemma}

\begin{proof}
By Definition~\ref{gla83}, a ball $B \in {\bf B}_X$ is a centered sphere in $(X,d)$
if and only if the equality
\begin{equation}
    \label{qp}
B = \{ x \in X \colon d(x,c_1)=r_1 \}\cup\{c_1\}
\end{equation}
holds for some $c_1 \in B$ and $r_1 \in \mathbb{R}^+.$
Similarly, $B$ is a centered sphere in $(B,d|_{B\times B})$
if and only if the equality
\begin{equation}\label{eq:BsphereB}
B=\{x\in B:d(x,c_2)=r_2\}\cup\{c_2\}
\end{equation}
holds
for some $c_2\in B$ and $r_2\in\mathbb{R}^+$.

Let \eqref{qp} hold. It follows from $B\subseteq X$ that
\begin{equation}\label{ss}
\{x\in X : d(x,c_1)=r_1\}\cup \{c_1\} \supseteq \{x\in B : d(x,c_1)=r_1\}\cup \{c_1\}.
\end{equation}
If the last inclusion is strict, then there is $p\in X$ satisfying
\[
d(p,c_1)=r_1 \quad \text{and} \quad p\notin B,
\]
that contradicts \eqref{qp}. Hence \eqref{qp} implies the equality.
\begin{equation*}
\{x\in X:d(x,c_1)=r_1\}\cup\{c_1\}
=
\{x\in B:d(x,c_1)=r_1\}\cup\{c_1\},
\end{equation*}
that gives us \eqref{eq:BsphereB}
with $r_2=r_1$ and $c_2=c_1$.
Thus if $B$ is a centered sphere in $(X,d)$, then $B$ also is  a
centered sphere in $(B,d|_{B\times B})$.

Suppose now that equality \eqref{eq:BsphereB} holds.
We claim that the equality
\begin{equation}\label{eq:equal3}
\{x\in X:d(x,c_2)=r_2\}\cup\{c_2\}
=
\{x\in B:d(x,c_2)=r_2\}\cup\{c_2\}
\end{equation}
is satisfied.

Similar to \eqref{ss}, we have  the inclusion
\begin{equation*}
\{x\in X:d(x,c_2)=r_2\}\cup\{c_2\}
\supseteq
\{x\in B:d(x,c_2)=r_2\}\cup\{c_2\}.
\end{equation*}
Consequently, equality~\eqref{eq:equal3} holds if we have
\begin{equation}\label{eq:equal4}
\{x\in X:d(x,c_2)=r_2\}\cup\{c_2\}
\subseteq
\{x\in B:d(x,c_2)=r_2\}\cup\{c_2\}.
\end{equation}
Equality \eqref{eq:BsphereB} implies $c_2\in B$. Hence, by Proposition \ref{dj}, we may assume  
\begin{equation}
    \label{toto}
B = B_r(c_2)=\{x\in X \colon d(x,c_2)<r\}
\end{equation}
 for some \( r > 0 \).
Equalities \eqref{eq:BsphereB} and \eqref{toto} imply the inequality
\(
    r_2<r,
\)
that together with \eqref{toto} gives us
\begin{equation}
    \label{toto1}
\{x\in X \colon d(x,c_2)=r_2\}\cup\{c_2\}\subseteq B.
\end{equation}
Inclusion \eqref{eq:equal4} follows from \eqref{eq:BsphereB} and \eqref{toto1}.

The proof is completed.
\end{proof}

If a subset $S$ of an ultrametric space $(X,d)$ is a centered sphere in itself but $S$ is not a singleton or an open ball in $(X,d)$, then $S$ may not be a centered sphere in $(X,d)$.

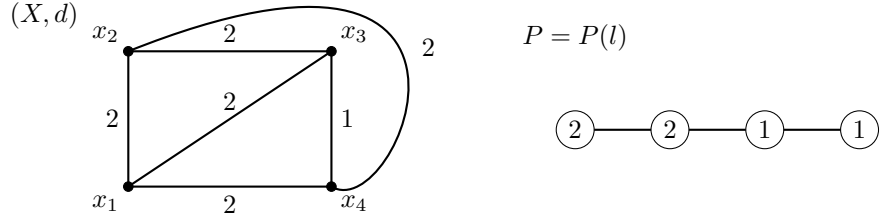
\begin{figure}[h]
\begin{center}
\begin{tikzpicture}[scale=1.8]

\coordinate (x1) at (0,0);
\coordinate (x2) at (0,1);
\coordinate (x3) at (1.5,1);
\coordinate (x4) at (1.5,0);

\draw[line width=0.8pt] (x1) -- (x2) -- (x3) -- (x4) -- cycle;
\draw[line width=0.8pt] (x2) .. controls (3.0,2.2) and (2.0,-0.3) .. (x4);
\draw[line width=0.8pt] (x1) -- (x3);

\filldraw[black] (x1) circle (1pt);
\filldraw[black] (x2) circle (1pt);
\filldraw[black] (x3) circle (1pt);
\filldraw[black] (x4) circle (1pt);

\node[below left]  at (x1) {$x_1$};
\node[above left]  at (x2) {$x_2$};
\node[above right] at (x3) {$x_3$};
\node[below right] at (x4) {$x_4$};

\node[left]  at ($(x1)!0.5!(x2)$) {$2$};
\node[above] at ($(x2)!0.5!(x3)$) {$2$};
\node[right] at ($(x3)!0.5!(x4)$) {$1$};
\node[below] at ($(x4)!0.5!(x1)$) {$2$};

\node[above] at ($(x1)!0.5!(x3)$) {$2$};

\node[above right] at (2.1,0.9) {$2$};
\node[above left] at (-0.3,1.1) {$(X,d)$};

\coordinate (p1) at (3.3,0.42);
\coordinate (p2) at (4.0,0.42);
\coordinate (p3) at (4.7,0.42);
\coordinate (p4) at (5.4,0.42);

\draw[line width=0.8pt] (p1) -- (p2) -- (p3) -- (p4);

\draw[fill=white] (p1) circle (0.14);
\draw[fill=white] (p2) circle (0.14);
\draw[fill=white] (p3) circle (0.14);
\draw[fill=white] (p4) circle (0.14);

\node at (p1) {$2$};
\node at (p2) {$2$};
\node at (p3) {$1$};
\node at (p4) {$1$};

\node at (3.3,1.1) {$P=P(l)$};

\end{tikzpicture}
\caption{The space $(X,d)$ is isometric to the ultrametric space $(V(P),d_l)$ generated by labeled path $P=P(l)$.}
\label{fig1}
\end{center}
\end{figure}

\begin{Example}
Let us consider the ultrametric space $(X,d)$ depicted by Figure~\ref{fig1}. Then $(X,d)\in {\bf UT}$ and the set $\{x_1,x_2\}$ is a centered sphere in itself
but this set is not a centered sphere in $(X,d)$.  
Indeed, the equality $d(x_1,x_2)=2$ and the definition of $(X,d)$ give us the equalities
\begin{equation*}
\{x\in X \colon d(x,x_1)=2\}\cup\{x_1\}
= \{x_1,x_2,x_3,x_4\}
= \{x\in X \colon d(x,x_2)=2\}\cup\{x_2\}.
\end{equation*}
\end{Example}

In what follows we denote  ${\bf Cs}_X$  the set of all centered spheres of an ultrametric space $(X,d)$.

\begin{Theorem}\label{efhik}
Let $(X,d)$ be an ultrametric space and let 
\begin{equation}
    \label{fesdi6-1}
D_p(X) := \{ d(p,x) : x \in X \}
\end{equation}
for each $p\in X$.
If the inclusion 
\begin{equation}
    \label{ed1}
{\bf B}_X \subseteq {\bf Cs}_X
\end{equation}
holds, then the set
$D_p(X)\cap [0,r)$ has the greatest element for all $p\in X$ and $r\in (0,\infty)$. Moreover, if $(X,d)\in {\bf UT}$ and the set $D_p(X)\cap [0,r)$ has the greatest element for all $p\in X$ and $r\in (0,\infty)$, then inclusion \eqref{ed1} holds.
\end{Theorem}

\begin{proof}
Let inclusion \eqref{ed1} hold, let $p$ be an arbitrary point of $X$, and let
$r\in(0,\infty)$.
We must show that the set $D_p(X)\cap[0,r)$ has the greatest element. 

Let us consider the open ball
\begin{equation}\label{oo1}
    B:=\{x\in X \colon d(x,p)<r\}.
\end{equation}
 By inclusion \eqref{ed1} the ball $B$ is a centered sphere in $(X,d)$, i.e., the equality
\[
B=\{x\in X\colon \ d(x,p_1)=r_1\}\cup\{p_1\}
\]
holds with some $p_1\in B$ and $r_1\in \mathbb R^+$. The last equality and the Lemma~\ref{nnhh}  with $A=B$ and $p=p_1$ imply 
\[
\operatorname{diam} B=r_1 \quad \text{and} \quad
 \operatorname{diam} B\in D(B),
\]
where $D(B)$ is the distance set of $(B,d|_{B\times B})$.

Consequently, the diametrical graph \(G_B\) of \((B,d|_{B\times B})\) is non-empty.
Theorem \ref{ref} implies that \(G_B\) is a complete multipartite graph.
Hence, by Definition \ref{scfr}, there is a point \(x\in B\) such that
\begin{equation*}
d(p,x)=\operatorname{diam} B.
\end{equation*}
Using \eqref{fesdi6-1}, \eqref{oo1} and formula \eqref{lokias} with \(A=B\), we see that \(\operatorname{diam} B\) is the greatest element of the set \(D_p(X)\cap[0,r)\) as required.

Let $(X,d)\in {\bf UT}$ and let the set $D_p(X)\cap [0,r)$ contain the greatest element for all $p\in X$ and $r\in (0,\infty)$.
Let us consider an arbitrary open ball $B^1\in {\bf B}_X$ with  a center $p_1\in X$ and a radius $r_1>0$,
\begin{equation}
    \label{nn1}
    B^1=\{x\in X \colon d(p_1,x)<r_1\}.
\end{equation}
To prove inclusion \eqref{ed1} we must show that $B^1$ is a centered sphere in $(X,d)$.

If $B^1$ is a singleton, then $B^1\in {\bf Cs}_X$ holds by Proposition~\ref{sat}. 

Suppose now that $B^1$ contains at least two points.
Since $(X,d)$ is an ${\bf UT}$-space, Lemma~\ref{scgkk} implies the existence of a labeled tree
$T^{1}=T^{1}(l_{1})$ such that
$(B^1,d|_{B^{1}\times B^{1}})$ and $(V(T^{1}),d_{l_1})$
are isometric.
Hence the ultrametric space $(B^{1},d|_{B^{1}\times B^{1}})$ also belongs to the class
${\bf UT}$.

By Theorem~\ref{shak1}, the membership
\((B^{1},d|_{B^{1}\times B^{1}})\in {\bf UT}
\)
implies  \(B^{1}\in  {\bf Cs}_{B^1}\)  if and only if
\begin{equation}
    \label{ng1}
C(B^{1})=\{0,\operatorname{diam} B^{1}\},
\end{equation}
where $C(B^1)$ is the center of distances of $(B^1, d|_{B^1\times B^1})$.
By Theorem~\ref{dggkq}, equality \eqref{ng1} is equivalent to 
\begin{equation}
\label{ng2}
\operatorname{diam} B^{1} \in D(B^{1}),
\end{equation}
where \(D(B^{1})\) is the distance set of the space \((B^{1},d|_{B^{1}\times B^{1}})\).
The space \((B^{1},d|_{B^{1}\times B^{1}})\) is a subspace of \((X,d)\).
Hence the inclusion
\[
D(B^{1}) \subseteq D(X)
\]
holds. Let $D_{p_1}(X)$ be the set defined by \eqref{fesdi6-1} with $p=p_1$, where $p_1$ is the center of the ball $B^1$. By supposition the set $D_{p_1}(X)\cap [0,r_1)$  has the greatest element $R_1$. Thus we have 
\begin{equation*}
    R_1\in D_{p_1} (X)\cap [0,r_1) \subseteq D(B^1).
\end{equation*}
Lemma~\ref{nnhh} and \eqref{nn1} imply $R_1=\operatorname{diam}B^1$.
  Thus \eqref{ng2} holds and, consequently, we have
\begin{equation}
    \label{ng3}
 B^1\in {\bf Cs}_{B^{1}},
\end{equation}
by Theorems \ref{dggkq} and \ref{shak1}.
Using Lemma \ref{dhiu} we may rewrite \eqref{ng3} as
\begin{equation}
    \label{ng4}
B^1 \in {\bf Cs}_X. 
\end{equation}
Thus \eqref{ng4} holds for each \(B^{1}\in {\bf B}_X\).
Inclusion \eqref{ed1} follows.

The proof is completed.
\end{proof}


\begin{Corollary}\label{fjqp}
Let $(X,d)\in {\bf UT}$ be totally bounded. Then the inclusion
\[
{\bf B}_X \subseteq {\bf Cs}_X
\]
holds.
\end{Corollary}

\begin{proof}
Let $p$ be a point of $X$, let
\[
D_p(X):=\{d(p,x): x\in X\},
\]
and let $r>0$.
The existence of the greatest element in the set $D_p(X)\cap [0,r)$ follows from Proposition \ref{zbhfd}.
Hence the inclusion
\(
{\bf B}_X \subseteq {\bf Cs}_X
\)
holds by Theorem~\ref{efhik}. 
\end{proof}

Since every compact metric space is totally bounded, Corollary~\ref{fjqp} implies the following.

\begin{Corollary}
Let $(X,d)\in {\bf UT}$ be compact. Then the inclusion
\begin{equation*}
{\bf B}_X \subseteq {\bf Cs}_X
\end{equation*}
holds. 
\end{Corollary}

The following example shows that the condition $(X,d)\in {\bf UT}$ cannot be omitted in Theorem~\ref{efhik} in the general case.

\begin{Example}[The $p$-adic ultrametric]
\label{nbv}
    Let $(X,d)$ be isometric to $({\mathbb Q}, d_p)$, where $d_p$ is the $p$-adic ultrametric defined in Example~\ref{ex1}  and let 
    \begin{equation*}
        D_q(X):=\{d(q,x)\colon x\in X\}
    \end{equation*}
    for $q\in X$.
    Then using formula~\eqref{ff1}, it is easy to see that the set $D_q(X)\cap[0,r)$ contains the greatest element for all $q\in X$ and $r\in (0,\infty)$. Since the space $({\mathbb Q}, d_p)$ does not contain isolated points, there are no open balls in $({\mathbb Q}, d_p)$ that are centered spheres in this space. Thus the inclusion ${\bf B}_{X}\subseteq {\bf Cs}_{X}$ is false. It should be noted that singletons are centered spheres by Proposition~\ref{sat} but such sets cannot be open balls if the space does not contain isolated points.
\end{Example}

It is interesting to compare Theorem~\ref{efhik} with the following result.

\begin{Theorem}\label{kmhg}
Let $(X,d)$ be an ultrametric space. Then the following statements are equivalent:

\begin{enumerate}[label=\textit{(\roman*)}, left=0pt]
\item The inclusion ${\bf B}_{X} \supseteq {\bf Cs}_{X}$ holds.
\item The ultrametric $d$ is equidistant.
\item The equality ${\bf B}_{X} = {\bf Cs}_{X}$ holds.
\end{enumerate}
\end{Theorem}

For the proof of Theorem \ref{kmhg} see Theorem~4.9 from \cite{DK2024DLPSSAG}.

\begin{Example} Let $(X,d)$ be an equidistant metric space. Then each non-empty subset $A$ of $X$ is a centered sphere in $(A,d|_{A\times A})$ but any centered sphere in $(X,d)$ is either a singleton or coincides with the set $X$.
\end{Example}

\section{Center of distances, diametrical graphs, and weak similarities}\label{sec5}

Using the concepts of diametrical graphs and star graphs we also can give
the necessary and sufficient conditions under which the equality
\[
C(X)=\{0,\operatorname{diam} X\}
\]
holds for ${\bf UT}$-spaces $(X,d)$.

\begin{Theorem}\label{kgcd}
Let $(X,d) \in {\bf UT}$ have at least two points and let $G_X$ be the diametrical graph of $(X,d)$. Then the following statements are equivalent:

\begin{enumerate}[label=\textit{(\roman*)}, left=0pt]
\item The equality
\begin{equation}
    \label{scg}
C(X)=\{0,\operatorname{diam} X\}
\end{equation}
holds.

\item $G_X$ is a complete multipartite graph such that at least one of its parts is a singleton.

\item $G_X$ has a spanning star subgraph.
\end{enumerate}
\end{Theorem}

\begin{proof}
    $(i) \Rightarrow (ii)$.
Let \eqref{scg} hold. Then the membership relation
\begin{equation}
    \label{kj}
\operatorname{diam} X\in D(X)
\end{equation}
is valid, because $C(X)$ is a subset of $D(X)$.

Definition \ref{dvhj}
and  membership relation \eqref{kj} imply that the graph $G_X$ is non-empty.
Consequently, this graph is complete multipartite by Theorem~\ref{ref}.
By Theorem~\ref{shak1}, equality~\eqref{scg} holds if and only if $X$ is a centered sphere in $(X,d)$ and $\operatorname{diam} X$ is the radius of this sphere.

Let $c$ be the center of the centered sphere $X$.
Then, using \eqref{ijgt} with $S=X$ and $r=\operatorname{diam} X$ we obtain the equality
\begin{equation*}
d(x,c)=\operatorname{diam} X
\end{equation*}
for each $x\in X\setminus\{c\}$.
Hence, by Definition~\ref{dvhj}, we have
\begin{equation*}
\{x,c\}\in E(G_X)
\end{equation*}
for each $x\in X\setminus\{c\}$.
Thus, the singleton $\{c\}$ is a part of $G_X$ by Definition~\ref{scfr}.
Consequently, the implication $(i)\Rightarrow(ii)$ is valid.

$(ii)\Rightarrow(i)$.
Let $(ii)$ hold.
Then $G_X$ is a complete multipartite graph and, consequently, $G_X$ is non-empty by Definition~\ref{scfr}.
Hence we can find two distinct points $x,y\in X$ such that
\begin{equation}
    \label{wswe}
\{x,y\}\in E(G_X).
\end{equation}
By Definition~\ref{dvhj}, membership \eqref{wswe} implies the equality
\[
d(x,y)=\operatorname{diam} X.
\]
Thus we have 
\begin{equation*}
    \operatorname{diam} X\in D(X),
\end{equation*}
 that implies statement $(i)$ by Theorem~\ref{dggkq}.

$(ii)\Rightarrow (iii)$. 
Let a singleton $\{c\}$ be a part of the complete multipartite graph $G_X$.
Then the membership relation
\[
\{x,c\}\in E(G_X)
\]
holds for all $x\in X\setminus\{c\}$ by Definition~\ref{scfr}.
Let $St_X$ be a subgraph of $G_X$ such that
\begin{equation}
    \label{dar}
V(St_X): = X
\end{equation}
and
\begin{equation}
        \label{alph00}
E(St_X):=\{\{x,c\}\colon x\in X\setminus\{c\}\}.
\end{equation}
Then $St_X$ is a star graph by Definition~\ref{efhjc} and, moreover,  $St_X$ is a spanning subtree of $G_X$ by \eqref{dar}.
Thus $(ii)\Rightarrow(iii)$ is a valid implication.

$(iii)\Rightarrow(ii)$.
Let $St_X$ be a spanning star subgraph of the diametrical graph $G_X$. 
Then \eqref{dar}
holds.
Since $(X,d)$ contains at least two points, $St_X$ is non-empty by Definition~\ref{efhjc}.
Thus $G_X$ also is non-empty by inclusion $St_X\subseteq G_X$.
Consequently, $G_X$ is complete multipartite by Theorem~\ref{ref}.
Definition~\ref{efhjc} and equality \eqref{dar} imply that there is $c\in X$ such that \eqref{alph00} holds 
for each $ x\in X\setminus\{c\}.$
Consequently the singleton $\{c\}$ is a part of the complete multipartite graph $G_X$. Thus $(iii)$ implies $(ii)$. 

The proof is completed.
\end{proof}

\begin{figure}[h]
\begin{center}
\begin{tikzpicture}[scale=2]

\coordinate (l1) at (-2.000,0.5);
\coordinate (l3) at (-1.000,0.5);
\coordinate (l2) at (-1.1339746,1.0);
\coordinate (l4) at (-1.1339746,0.0);

\draw[line width=0.8pt] (l1) -- (l2);
\draw[line width=0.8pt] (l1) -- (l3);
\draw[line width=0.8pt] (l1) -- (l4);

\filldraw (l1) circle (1pt);
\filldraw (l2) circle (1pt);
\filldraw (l3) circle (1pt);
\filldraw (l4) circle (1pt);

\node[below left]  at (l1) {$x_1$};
\node[above left]  at (l2) {$x_2$};
\node[above right] at (l3) {$x_3$};
\node[below right] at (l4) {$x_4$};

\node[above left] at (-1.9,1.15) {$St_X$};

\coordinate (x1) at (0,0);
\coordinate (x2) at (0,1);
\coordinate (x3) at (1.5,1);
\coordinate (x4) at (1.5,0);

\draw[line width=0.8pt] (x1)--(x2)--(x3);
\draw[line width=0.8pt] (x1)--(x4);

\draw[line width=0.8pt] (x1)--(x3);
\draw[line width=0.8pt] (x2)--(x4);

\filldraw (x1) circle (1pt);
\filldraw (x2) circle (1pt);
\filldraw (x3) circle (1pt);
\filldraw (x4) circle (1pt);

\node[below left]  at (x1) {$x_1$};
\node[above left]  at (x2) {$x_2$};
\node[above right] at (x3) {$x_3$};
\node[below right] at (x4) {$x_4$};

\node[above left] at (-0.25,1.15) {$G_X$};

\end{tikzpicture}

\caption{A spanning star graph $St_X$ of the diametrical graph $G_X$ of the space $(X,d)$ depicted by Figure~\ref{fig1}.}
\label{fig2}
\end{center}
\end{figure}
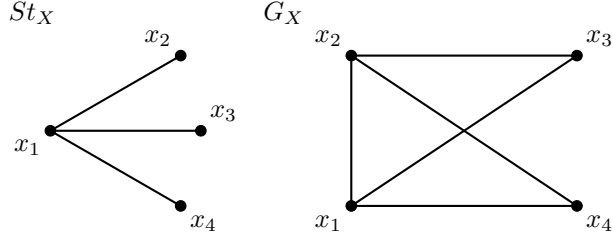

\begin{Remark}
    Analyzing the proof of Theorem~\ref{kgcd} we can see that the equivalence $(ii) \Leftrightarrow (iii)$ is valid for arbitrary complete multipartite graph $G$.
Consequently, Theorem~\ref{kgcd} gives us necessary and sufficient conditions under which complete multipartite graphs have  spanning star subgraphs. In connection with this result, we note that all connected graphs have spanning subtrees. See, for example, Proposition~14 in Chapter~1 of book~\cite{SerreTrees1980}.
\end{Remark}

Theorems \ref{rgjkjs24} and \ref{dggkq} give us the following.

\begin{Corollary}
\label{shak}
Let $(X,d)\in {\bf  UT}$ have at least two points and let $C(X)$ be the center of distances of $(X,d)$.
Then the equality
\begin{equation}
    \label{ks1}
C(X)=\{0\}
\end{equation}
holds if and only if $(X,d)$ is weakly similar to an unbounded ultrametric space.
\end{Corollary}

\begin{proof}
Let $D(X)$ be the distance set of $(X,d)$.
Then, by Theorem~\ref{dggkq}, equality \eqref{ks1} holds iff
\[
\operatorname{diam} X \notin D(X).
\]
It implies the desired result by Theorem~\ref{rgjkjs24}.
\end{proof}

\begin{Remark}
    \label{oiuw}
Example~\ref{ex1} shows that equality \eqref{ks1} may not hold if $(X,d)\notin {\bf UT}$.
\end{Remark}

\section{Conclusion. Expected results}\label{sec6}

Let $(X,d)$ be a metric space.
A {\it closed ball} in $(X,d)$ is a subset $\overline{B}_r(c)$ of the set $X$ defined as
\[
\overline{B}_r(c):=\{x\in X \colon d(c,x)\leq r\},
\]
where $c\in X$ and $r\in\mathbb{R}^+$.

Let us denote by $\overline{\bf B}_X$ the set of all closed balls of $(X,d)$.

The following conjecture is a counterpart of Theorem~\ref{efhik}.

\begin{Conjecture}
    \label{efhik-5}
Let $(X,d)$ be an ultrametric space and let 
\begin{equation*}
D_p(X) := \{ d(p,x) : x \in X \}
\end{equation*}
for each $p\in X$. 
If the inclusion 
\begin{equation}
    \label{ed1-5}
\overline{{\bf B}}_X \subseteq {\bf Cs}_X
\end{equation}
holds, then the set
$D_p(X)\cap [0,r]$ has the greatest element for all $p\in X$ and $r\in \mathbb R^+$. Moreover, if $(X,d)\in {\bf UT}$ and the set $D_p(X)\cap [0,r]$ has the greatest element for all $p\in X$ and $r\in \mathbb R^+$, then inclusion \eqref{ed1-5} holds.
\end{Conjecture}

Let $(X_3,d)$ be the three-point ultrametric space depicted in Figure~\ref{fig3} below.

\begin{figure}[h]
\begin{center}
\begin{tikzpicture}[scale=2]

\coordinate (a) at (0.3,0);
\coordinate (b) at (1.3,0);
\coordinate (c) at (0.8,1.32); 

\draw[line width=0.8pt] (a) -- (b);
\draw[line width=0.8pt] (b) -- (c);
\draw[line width=0.8pt] (c) -- (a);

\filldraw (a) circle (1pt);
\filldraw (b) circle (1pt);
\filldraw (c) circle (1pt);

\node[below left]  at (a) {$x_1$};
\node[above]       at (c) {$x_2$};
\node[below right] at (b) {$x_3$};

\node[yshift=-6pt] at ($(a)!0.5!(b)$) {$1$};
\node[xshift=8pt,yshift=6pt] at ($(b)!0.5!(c)$) {$2$};
\node[xshift=-8pt,yshift=6pt] at ($(c)!0.5!(a)$) {$2$};

\node[above right] at (1.55,1.15) {$(X_3,d)$};

\end{tikzpicture}

\caption{A non-equidistant three-point ultrametric space $(X_3,d)$.}
\label{fig3}
\end{center}
\end{figure}
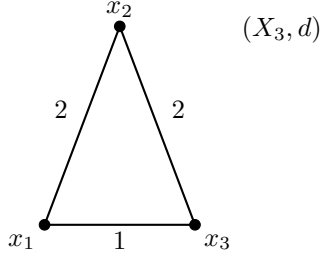

Then it is easy to show that every non-empty subset of $X_3$ is a centered sphere in $(X_3,d)$.

\begin{Conjecture}\label{hol}
Let $(Y,\rho)$ be an ultrametric space such that $|Y|\geq 3$. Then the following statements are equivalent:
\begin{enumerate}[label=\textit{(\roman*)}, left=0pt]
\item Each non-empty subset of the set $Y$ is a centered sphere in $(Y,\rho)$.
\item The set $Y$ contains exactly three points and $(Y,\rho)$ is weakly similar to $(X_3,d)$.
\end{enumerate}
\end{Conjecture}

The following conjecture was motivated by some calculations with \(p\)-adic ultrametric with \(p=2\) (see Example~\ref{ex1}).

\begin{Conjecture}\label{con3}
Let $n \ge 1$ be an integer number, 
let $\log_2(n)$ be the binary logarithm of $n$, and let $\lfloor \log_2(n) \rfloor$ be the integer part of $\log_2(n)$.
Then the inequality
\begin{equation*}
|C(X)| \leq 1 + \lfloor \log_2 (n) \rfloor
\end{equation*}
holds for each ultrametric space $(X,d)$ with $|X|=n$.
Moreover, there exists an ultrametric space $(Y,\rho)$ such that $|Y|=n$
and
\[
|C(Y)| = 1 + \lfloor \log_2(n) \rfloor .
\]
\end{Conjecture}

To see why a binary logarithm bound is expected in Conjecture~\ref{con3}, we only note that the following simple statement is valid.

Let \(\{X_1, X_2, \dots, X_k\}\) be a partition of a finite set  \(X\) on disjoint subsets \(X_1, X_2, \dots, X_k\) with \(k \geq 2\).
Then there is a partition \(\{X_1^*, X_2^*\}\) of \(X\) such that
\[
\min\{|X_1^*|, |X_2^*|\} \geq \min\{|X_1|, \dots, |X_k|\}.
\]

\hfill

\section*{Funding}

The first author was supported by grant 367319 of the Research Council of Finland.

\bibliographystyle{unsrt} 
\bibliography{bib2020.07}

@Book{Bachman1964,
  author    = {Bachman, G.},
  publisher = {Academic Press},
  title     = {{Introduction to $p$-Adic Numbers and Valuation Theory}},
  year      = {1964},
  address   = {New York/London},
}

@Article{BPW2018,
  author  = {Bielas, W. and  Plewik, S. and  Walczyńs, M.},
  title   = {{On the center of distances}},
  journal = {European Journal of Mathematics},
  year    = {2018},
  volume  = {4},
  pages   = {687--698},
}

@Article{BDK2022TAG,
  author  = {Bilet, V. and Dovgoshey, O. and Kononov, Y.},
  title   = {{Ultrametrics and Complete Multipartite Graphs}},
  journal = {Theory and Applications of Graphs},
  year    = {2022},
  volume  = {9},
  number  = {1},
  pages   = {Article 8},
}

@Article{DLPS2008TaiA,
  author  = {Delhomm\'{e}, C. and Laflamme, C. and Pouzet, M. and Sauer, N.},
  title   = {{Indivisible ultrametric spaces}},
  journal = {Topology and its Applications},
  year    = {2008},
  volume  = {155},
  number  = {14},
  pages   = {1462--1478},
}

@book{Dez-Dez,
 author = {Deza, M. M. and Deza, E.},
 title = {Encyclopedia of distances},
 edition = {4th},
 isbn = {978-3-662-52843-3; 978-3-662-52844-0},
 year = {2016},
address		= {Berlin},
 publisher = {Springer},
}

@Book{Diestel2005,
  author    = {Diestel, R.},
  publisher = {Springer},
  title     = {{Graph Theory}},
  year      = {2005},
  address   = {Berlin},
  edition   = {{Third}},
  series    = {Graduate Texts in Mathematics},
  volume    = {173},
}

@Article{DDP2011pNUAA,
  author   = {Dordovskyi, D. and Dovgoshey, O. and Petrov, E.},
  title    = {Diameter and diametrical pairs of points in ultrametric spaces},
  journal  = {$p$-Adic Numbers, Ultrametric Analysis and Applications},
  year     = {2011},
  volume   = {3},
  number   = {4},
  pages    = {253--262},
  fjournal = {p-Adic Numbers, Ultrametric Analysis, and Applications},
}

@Article{Dov2019pNUAA,
  author  = {Dovgoshey, O.},
  title   = {{Finite ultrametric balls}},
  journal = {p-adic Numbers Ultrametr. Anal. Appl.},
  year    = {2019},
  volume  = {11},
  number  = {3},
  pages   = {177--191},
  issn    = {2070-0466},
}

@Article{Dov2020TaAoG,
  author  = {Dovgoshey, O.},
  journal = {Theory and Applications of Graphs},
  title   = {Isomorphism of trees and isometry of ultrametric spaces},
  year    = {2020},
  note    = {Article 3},
  number  = {2},
  volume  = {7},
}

@Article{DCR2025OJAC,
  author  = {Dovgoshey, O. and Cantor, O. and Rovenska, O.},
  title   = {Compact ultrametric spaces generated by labeled star graphs},
  journal = {Online Journal of Analytic Combinatorics},
  year    = {2025},
  volume  =  {20},
  pages   = {1--20},
}

@Article{DK2024DLPSSAG,
  author  = {Dovgoshey, O. and Kostikov, A.},
  journal = {Journal of Mathematical Sciences},
  title   = {{Locally finite ultrametric spaces and labeled trees}},
  year    = {2023},
  number  = {5},
  pages   = {614--637},
  volume  = {276},
}

@Article{DK2022LTGCCaDUS,
  author  = {Dovgoshey, O. and K\"{u}\c{c}\"{u}kaslan, M.},
  journal = {Annals of Combinatorics},
  title   = {Labeled trees generating complete, compact, and discrete ultrametric spaces},
  year    = {2022},
  pages   = {613--642},
  volume  = {26},
}

@Article{DP2013AMH,
  author   = {Dovgoshey, O. and Petrov, E.},
  title    = {{Weak similarities of metric and semimetric spaces}},
  journal  = {Acta Mathematica Hungarica},
  year     = {2013},
  volume   = {141},
  number   = {4},
  pages    = {301--319},
  fjournal = {Acta Mathematica Hungarica},
}

@Article{DR2025USGbLSG,
  author  = {Dovgoshey, O. and Rovenska, O.},
  journal = {Journal of Mathematical Sciences},
  title   = {Ultrametric spaces generated by labeled star graphs},
  year     = {2025},
  volume   = {228},
  number   = {2},
  pages    = {182--198},
}

@Article{DS2022TRoUCaS,
  author  = {Dovgoshey, O. and Shcherbak, V.},
  journal = {Topology and its Applications},
  title   = {The range of ultrametrics, compactness, and separability},
  year    = {2022},
  pages   = {107899},
  volume  = {305},
}

@article{DV2025JMS,
  author    = {O. Dovgoshey and V. Vito},
  title     = {Totally bounded ultrametric spaces generated by labeled rays},
  journal   = {Applied General Topology},
  year      = {2025},
  volume    = {26},
  number    = {1},
  pages     = {163--182}
}

@Article{GH1961S,
  author  = {Gomory, R. E. and Hu, T. C.},
  title   = {Multi-terminal network flows},
  journal = {SIAM},
  year    = {1961},
  volume  = {9},
  number  = {4},
  pages   = {551--570},
}

@book{Neumann1935,
  author    = {von Neumann, J.},
  title     = {Charakterisierung des Spektrums eines Integraloperators},
  publisher = {Hermann},
  address   = {Paris},
  year      = {1935}
}

@Article{NPT2025,
  author  = {Nowakowski, P. and  Prus-Wiśniowski, F. and Turoboś, F},
  title   = {{Spectre operator, achievement sets and sets of $P$-sums in a hyperspace of compact sets}},
  journal = {arXiv:2512.11803},
  year    = {2025},
}

@Article{PD2014JMS,
  author   = {Petrov, E. and Dovgoshey, A.},
  title    = {{On the {G}omory-{H}u inequality}},
  journal  = {Journal of Mathematical Sciences},
  year     = {2014},
  volume   = {198},
  number   = {4},
  pages    = {392--411},
}

@Book{Sch1985,
  title     = {{Ultrametric Calculus. An Introduction to p-Adic Analysis}},
  publisher = {Cambridge University Press},
  year      = {1985},
  author    = {Schikhof, W. H.},
}

@book{SerreTrees1980,
  author    = {Serre, Jean-Pierre},
  title     = {Trees},
  translator= {Stillwell, John},
  publisher = {Springer-Verlag},
  address   = {Berlin--Heidelberg--New York},
  year      = {1980},
}

\end{document}